\documentclass[12pt,reqno,a4paper]{amsart}

\usepackage{amsmath,amsfonts,amsthm,amssymb,amsxtra,enumerate,mathtools,mathabx}
\usepackage[normalem]{ulem}
\usepackage[utf8]{inputenc}
\usepackage[T1]{fontenc}
\usepackage{lmodern}
\usepackage{bbm}
\usepackage{mathrsfs}
\usepackage{csquotes}
\usepackage{enumerate}
\usepackage{comment}
\usepackage{placeins}
\usepackage{pgfplots}
\pgfplotsset{compat=newest}
\usepgfplotslibrary{fillbetween}
\usepackage{tikz}
\usepackage{enumitem}

\newtheorem{theorem}{Theorem}
\newtheorem{proposition}[theorem]{Proposition}
\newtheorem{lemma}[theorem]{Lemma}
\newtheorem{corollary}[theorem]{Corollary}

\theoremstyle{definition}

\theoremstyle{remark}

\newtheorem{remark}[theorem]{Remark}

\newcommand{\dd}{\, \mathrm{d}}

\renewcommand{\epsilon}{\varepsilon}

\newcommand{\N}{\mathbb{N}}
\newcommand{\norm}[2][]{{\left\|#2\right\|_{#1}}} 
\newcommand{\abs}[2][]{{\left\vert#2\right\vert_{#1}}} 

\renewcommand{\phi}{\varphi}
\newcommand{\R}{\mathbb{R}}

\newcommand{\Sph}{\mathbb{S}}

\newcommand{\lapr}{-\frac{\mathrm{d}^2}{\mathrm{d}r^2}}

\usepackage{hyperref}
\begin{document}
	
	\title[Schr\"odinger operators with oscillating potentials]{Negative spectrum of Schr\"odinger Operators with Rapidly Oscillating Potentials}

    \author{Larry Read} 
	\address{Larry Read, Mathematisches Institut, Ludwig-Maximilans Universit\"at M\"unchen, Theresienstr. 39, 80333 München, Germany}
	\email{read@math.lmu.de}
\maketitle
\begin{abstract}
    For Schr\"odinger operators with potentials that are asymptotically homogeneous of degree $-2$, the size of the coupling determines whether it has finite or infinitely many negative eigenvalues. In the latter case the asymptotic accumulation of these eigenvalues at zero has been determined by Kirsch and Simon. 

    A similar regime occurs for potentials which are not asymptotically monotone, but oscillatory. In this case when the ratio between the amplitude and frequency of oscillation is asymptotically homogeneous of degree $-1$ the coupling determines the finiteness of the negative spectrum. We present a new proof of this fact by making use of a ground state representation. As a consequence of this approach we derive an asymptotic formula analogous to that of Kirsch and Simon.
\end{abstract}
\section{Introduction}
For the self-adjoint Schr\"odinger operator ${-\Delta-V}$ in $L^2(\R^d)$, the rate of decay of $V$ near infinity determines whether its negative spectrum is finite. It is known (see e.g. \cite{Frank2022SchrodingerInequalities}) that if there exists $R<\infty$ such that
\begin{align} \label{eqn:finitecond}
    V(x)\leq
        \frac{(d-2)^2}{4\abs{x}^2}+
        \frac{1}{4\abs{x}^2(\ln \abs{x})^2} \ \text{  for all }\abs{x}\geq R
\end{align}
then the number of its negative eigenvalues is finite. Conversely, if $V$ has slower decay where there are $\varepsilon>0$ and $R<\infty$ such that
\begin{align}\label{eqn:infinitecond}
     V(x)\geq
        \frac{(d-2)^2}{4\abs{x}^2}+
      \frac{(1+\varepsilon)}{4\abs{x}^2(\ln \abs{x})^2} \ \text{  for all }\abs{x}\geq R
\end{align}
then the operator has infinitely many negative eigenvalues, accumulating at zero. 

In both regimes, there has been much success in bounding the number or determining the exact asymptotic accumulation of these negative eigenvalues. A standard phase space heuristic suggests that the number of eigenvalues below $-E\leq 0$, which we denote by $N_E(V)$, should coincide with the volume of 
\begin{equation*} \label{eqn:phasespaceOMEGAE}
    \Omega_E(V)=\{(x,\xi)\in \R^d\times\R^d\colon |2\pi\xi|^2-V(x)<-E\}.
\end{equation*}
Many results are semiclassical, corroborating this prediction. Most notable is the Cwikel-Lieb-Rozenblum inequality \cite{CwikelsInequality,Lieb1976BoundsOperators,Rozenbljum1972DistributionOperators}, which states that for $d\geq 3$
\begin{equation}
    N_0(V)\lesssim_d \abs{\Omega_0(V)}.
\end{equation}
Furthermore, according to \cite[Theorem XIII.82]{Reed1978IV:Operators}, if $V(x)=\lambda\abs{x}^{-2+\varepsilon}(1+o(1))$ towards infinity, then
\begin{equation*}
    N_{E}(V)=\abs{\Omega_E(V)}(1+o(1)) \text{ as }E\downarrow 0.
\end{equation*}

In \cite{Kirsch1988CorrectionsOperators} Kirsch and Simon found analogous asymptotics in the borderline case, where the potential satisfies
\begin{equation*}
    V(x)=\lambda\abs{x}^{-2}(1+o((\ln \abs{x})^{-2})) \text{ as }\abs{x}\rightarrow \infty.
\end{equation*}
In this scenario, despite the finiteness of $N_0(V)$ for small couplings $\lambda$, the volume of $\Omega_E(V)$ diverges logarithmically. Consequently, adjustments to the conventional phase space volume are necessary. The main result in \cite{Kirsch1988CorrectionsOperators}, with its subsequent improvement by Hassell and Marshall in \cite{Hassell2008Eigenvalues-2}, states that
\begin{align}\label{eqn:KirschSimon}
N_E(V)=(2\pi)^{-1}\abs{\ln E}\sum_{k=1}\sqrt{\left(\lambda-\frac{(d-2)^2}{4}- \Lambda_k\right)_+}+O(1) \text{ as }E\downarrow 0,
\end{align}
where $\{\Lambda_k\}_{k=1}$ are the eigenvalues of the spherical Laplacian, $-\Delta_{\Sph^{d-1}}$ in $L^2(\Sph^{d-1})$. 

Thus far we've noted that the nature of the negative spectrum bifurcates according to whether the potential lies above or below a critical function with regular decay. However, Willett \cite{Willett1969OnEquations} and Wong \cite{Wong1969OscillationCoefficients} have demonstrated that in one dimension the potential \[V(x)=\frac{\lambda \sin x}{x}\] generates only finitely many negative eigenvalues for $\abs{\lambda}\leq 1/\sqrt{2}$ and infinitely many for $\abs{\lambda}>1/\sqrt{2}$. The significance of the coupling constant for this potential, and the fact its oscillatory nature supports much slower decay, isn't predicted by semiclassical heuristics.  

In this paper we are concerned with the effect of oscillatory behaviour on $N_E(V)$ as $E\downarrow 0$. To this end, our main result is an analogue of \eqref{eqn:KirschSimon} for a large class of potentials which exhibit this critical coupling behaviour. 
\begin{theorem}\label{thm:main}
    Let $V\in L^\infty_{\mathrm{loc}}(\R^d)$ satisfy 
    \begin{equation}\label{eqn:criticaloscpot}
        V(x)=\lambda|x|^{\alpha-2}\sin |x|^\alpha+o((|x|\ln\abs{x})^{-2}) \text{ as }|x|\rightarrow \infty
    \end{equation}
    for any $\alpha>0$ and $\lambda\in \R$. Then
    \begin{equation}\label{eqn:mainthmform}
        N_E(V)=(2\pi)^{-1}\abs{\ln E}\sum_{k=1}\sqrt{\left(\frac{\lambda ^2}{2 \alpha ^2}-\frac{(d-2)^2}{4}- \Lambda_k \right)_+}+O(1) \text{ as }E\downarrow 0,
    \end{equation}
    where $\{\Lambda_k\}_{k=1}$ denote the eigenvalues of $-\Delta_{\Sph^{d-1}}$ in $L^2(\Sph^{d-1})$. In particular if $|\lambda|\leq \alpha\abs{d-2}/\sqrt{2}$ then $N_0(V)$ is finite.
\end{theorem}
In this result, both potentials with slow decay and rapid growth are permissible, subject to the rate of oscillation at infinity. Comparable to the regular conditions \eqref{eqn:finitecond} and \eqref{eqn:infinitecond}, we will show that if a potential oscillates or decays faster, or grows slower, than the critical \eqref{eqn:criticaloscpot} then $N_0(V)$ will be finite (or infinite conversely). 

The difficulty in studying such operators comes from the absence of simple variational methods. Standard upper or lower bounds on the potential don't capture the intricate interactions between its attractive and repulsive parts. To overcome this, we use an idea of Combescure and Ginibre \cite{Combescure1976SpectralPotentials,Combescure1980SpectralPotentials}. Specifically, we leverage a ground state representation to transform our operator to one with a purely attractive potential that subsumes the original repulsive components. It is a result of this transformation that the leading term in \eqref{eqn:mainthmform} is $O(\lambda^d)$ whilst all the aforementioned results scale with the semiclassical $O(\lambda^{d/2})$.

We note that in \cite{Raikov2016DiscretePotentials} Raikov has determined the result above for $\alpha=1$. They consider a larger class of potentials which consist of the product of an almost periodic function and a function which decays asymptotically with degree $-1$. Similarly the author reduces the operator to one with an effective attractive potential. However, the case of rapid (or slow) rates of oscillation are illusive to the approach in \cite{Raikov2016DiscretePotentials}. Our results refine this, in the radial regime, from a somewhat different perspective. 

A direct consequence of the formula in Theorem \ref{thm:main}, is that the negative eigenvalues which accumulate to zero can be characterised, up to a constant factor. 
\begin{corollary}\label{cor:main}
    Let $V$ be as in Theorem \ref{thm:main} and denote by $\{\lambda_k(V)\}_{k=1}$ the negative $-\Delta-V$ in $L^2(\R^d)$, in ascending order. Then there exists $C,c>0$ and $K\in \N$ such that
    \begin{equation*}
       c\exp\left(-\frac{k}{M}\right) \leq \abs{\lambda_k(V)}\leq C\exp\left(-\frac{k}{M}\right) \text{ for all } k\geq K,
    \end{equation*}
    where $M=M(\lambda,\alpha,d)$ is the coefficient of $\abs{\ln E}$ in \eqref{eqn:mainthmform}.
\end{corollary} 
The plan for the paper is as follows. In Section \ref{sec:hardy} we recall how Hardy's inequality can be applied to obtain the conditions \eqref{eqn:finitecond} and \eqref{eqn:infinitecond}. Then in Section \ref{sec:groundstate} we introduce the ground state representation, and apply it to determine the critical nature of the oscillating potentials \eqref{eqn:criticaloscpot}. Finally, in Section \ref{sec:mainproof} we use this representation with the method of Hassell and Marshall from \cite{Hassell2008Eigenvalues-2} to prove Theorem \ref{thm:main}.

\section{Hardy's inequality and finiteness of the negative spectrum} \label{sec:hardy}

Hardy's inequality serves as an immediate precursor to condition \eqref{eqn:finitecond}, establishing that potentials bounded everywhere by a reduced form of \eqref{eqn:finitecond} yield no negative eigenvalues. It states that 
\begin{equation}\label{eqn:hardyweight}
    \int_{\R^d} \frac{(d-2)^2\abs{u(x)}^2}{4\abs{x}^2}\dd x\leq \int_0^\infty \abs{\nabla u (x)}^2\dd x \text{ for all }u\in C_0^\infty(\R^d\backslash\{ 0\}).
\end{equation}
To extend this to condition \eqref{eqn:finitecond}, which addresses the asymptotic behavior of potentials, one may apply specialized variants of the inequality and employ bracketing techniques. We recall this argument, which can be found in \cite{Frank2022SchrodingerInequalities,Frank2022AnInequality}.

Suppose that the operator $-\Delta-V$ in $L^2(\R^d)$ is well-defined in the form sense with form domain $H^1(\R^d)$. Then through bracketing, we can reduce it by imposing Robin boundary conditions along the boundary of the ball $B_R=\{x\colon \abs{x}<R\}$, for any $R>0$. Namely, for any $\sigma\in \R$, take $H^{-}_{\sigma,R}$ and $H^{+}_{\sigma,R}$ to be the unique operators corresponding to the quadratic forms 
\begin{equation} \label{eqn:robinops}
    \begin{split}
    &\int_{B_R}\abs{\nabla u}^2\dd x-\sigma \int_{\partial B_R}\abs{u(y)}^2\dd \nu(y), \text{ and }\\
    &\int_{\overline{B_R}^c}\abs{\nabla u}^2\dd x+\sigma \int_{\partial B_R}\abs{u(y)}^2\dd \nu(y)
    \end{split}
\end{equation}
with form domains $H^1(B_R)$ and $H^1(\overline{B_R}^c)$, respectively, where $d\nu$ denotes the surface measure on $\partial B_R$. Then the operator $-\Delta-V$ is bounded from below, in the form sense, by the direct sum $H_{\sigma,R}^{-}\oplus H_{\sigma,R}^{+}$. In particular
\begin{equation}\label{eqn:bracketingRobin}
    N_E(V)\leq N_E(H_{\sigma,R}^{-})+N_E(H_{\sigma,R}^{+}) \ \text{ for any }-E\leq 0,
\end{equation}
where we take $N_E(\ \cdot\ )$ to count the eigenvalues of the enclosed operator below $-E$. Under fairly general assumptions on $V$, the spectrum of $H_{\sigma,R}^{-}$ is purely discrete and thus $N_0(H_{\sigma,R}^{-})<\infty$. As a result, the finiteness of $N_0(V)$ follows from that of $N_0(H^{+}_{\sigma,R})$. 

If we can choose $R$ such that \[V(x)\leq \frac{(d-2)^2}{4\abs{x}^2}\text{ for all }\abs{x}\geq R,\] then we can invoke a Robin variant of Hardy's inequality due to Kova\v{r}\'{i}k and Laptev \cite{Kovarik2012HardyLaplacians}. It states that if $\sigma\geq 1/2R$, then for any $u\in H^1(\overline{B_R}^c)$
\begin{equation}\label{eqn:RobinHardyKovarik}
    \int_{\overline{B_R}^c}\frac{(d-2)^2\abs{u}^2}{4\abs{x}^2}\dd x\leq \int_{\overline{B_R}^c}\abs{\nabla u(x)}^2\dd x+\sigma\int_{\partial B_R}\abs{u(y)}^2\dd \nu(y).
\end{equation}
Thus, after selecting an appropriately large $\sigma$ it transpires that $N_0(H_{\sigma,R}^{+})=0$ and we deduce from \eqref{eqn:bracketingRobin} that $N_0(V)$ is finite. 

To include the logarithmic term in \eqref{eqn:finitecond}, a coordinate transformation can be applied to \eqref{eqn:RobinHardyKovarik}. This inequality is established below, in a one-dimensional setting, featuring general weights that play a role in Section \ref{sec:groundstate}.
\begin{lemma}\label{lem:robinhardylog}
    Let $\rho\in \R$ and $R>1$. Then for any $\sigma\geq R^{\rho-1}((\ln R)^{-1}+(1-\rho))/2$,
    \begin{equation*}
         \int_R^\infty \left(\frac{(\rho-1)^2}{4r^2}+\frac{1}{4r^2(\ln r)^2}\right)\abs{u(r)}^2 r^{\rho}\dd r\leq \int_R^\infty \abs{u^\prime(r)}^2 r^\rho \dd r+\sigma \abs{u(R)}^2
    \end{equation*}
    for all $u\in H^1((R,\infty),r^{\rho}\dd r)$. 
\end{lemma}
\begin{proof}
    Let $v(r)\coloneqq r^{(\rho-1)/2}u(r)$, then proving the stated inequality is equivalent to showing that
    \begin{align*}
        \int_R^\infty \frac{\abs{v(r)}^2}{4r^2(\ln r)^2} r\dd r\leq \int_R^\infty \abs{ v^\prime(r)}^2 r\dd r+(\sigma R^{1-\rho}+(1-\rho)/2)\abs{v(R)}^2.
    \end{align*}
    \newpage
    Now making the substitution $\widetilde{v}(t)\coloneqq v(e^t)$ we see that this changes to 
    \begin{align*}
        \int_{\ln R}^\infty \frac{ \abs{\widetilde v(t)}^2}{4t^2} \dd t\leq \int_{\ln R}^\infty \abs{\widetilde v^\prime}^2\dd t+(\sigma R^{1-\rho}+(1-\rho)/2)\abs{\widetilde v(\ln R)}^2
    \end{align*}
    but this inequality, under the conditions imposed on $\sigma$, is exactly that of Kova\v{r}\'{i}k and Laptev \eqref{eqn:RobinHardyKovarik}.
\end{proof}
Now we address the converse claim, \eqref{eqn:infinitecond}, asserting that slow asymptotic decay leads to infinitely many negative eigenvalues. 

The operator $-\Delta-V$ can be bounded from above by imposing Dirichlet boundary conditions on $\partial B_R$. That is, if we let $H^{-}_{\infty,R}$ and $H^{+}_{\infty,R}$ correspond to the forms \eqref{eqn:robinops} with respective domains $H^1_0(B_R)$ and $H^1_0(\overline{B_R}^c)$, then 
\begin{equation*}
    N_E(V)\geq N_E(H^{-}_{\infty,R})+N_E(H^{+}_{\infty,R}) \ \text{ for any }-E\leq 0.
\end{equation*}
Under the condition that $R<\infty$ and $\varepsilon>0$ can be selected so that 
\begin{equation*}
    V(x)\geq\frac{(d-2)^2}{4\abs{x}^2}+\frac{(1+\varepsilon)}{4\abs{x}^2\left(\ln\abs{x}\right)^2} \text{ for all } \abs{x}\geq R,
\end{equation*}
the subsequent lemma establishes an infinite dimensional subspace of $L^2(\overline{B_R}^c)$ corresponding to the negative spectrum of $H^{+}_{\infty,R}$, and thus $N_0(V)=\infty$. 
\begin{lemma}\label{lem:reversehardy}
    Let $\rho\in \R$ and $R\geq 0$. Then for any $\varepsilon>0$ there exists an infinite sequence of $\{u_k\}_{k=1}\subset H_0^1((R,\infty), r^\rho\dd r)$ which are orthonormal in $L^2((R,\infty),r^\rho\dd r)$ and satisfy
    \begin{equation*}
         \int_R^\infty \left(\frac{(\rho-1)^2}{4r^2}+\frac{(1+\varepsilon)}{4r^2(\ln r)^2}\right)\abs{u_k(r)}^2 r^{\rho}\dd r> \int_R^\infty \abs{u_k^\prime(r)}^2 r^\rho \dd r.
    \end{equation*}
\end{lemma}
\begin{proof}
    Note that we can carry out the same change of coordinates used in Lemma \ref{lem:robinhardylog}. Then it is sufficient to show that there is an infinite sequence of bounded and compactly supported functions $v_k$, with disjoint supports, such that 
    \begin{align}\label{eqn:lemproof}
        \int_{\ln R}^\infty \frac{(1+\varepsilon)\abs{ v_k(t)}^2}{4t^2}\dd t> \int_{\ln R}^\infty \abs{  v_k^\prime(t)}^2 \dd t.
    \end{align}
    Moreover, it is only necessary to determine this with $R=1$ for one function supported in $(0,\infty)$. To see this, note that if $\widetilde v$ is such that $h[\widetilde v]<0$, where
    \begin{equation*}
        h[v]\coloneqq \int_{0}^\infty \abs{ v^\prime(t)}^2- \frac{(1+\varepsilon)\abs{v(t)}^2}{4t^2}\dd t,
    \end{equation*}
    then for any $\kappa>0$, under the unitary operator $\mathcal{U}_\kappa v(t)\coloneqq \kappa^{-1/2} v(\kappa^{-1} t)$,
    \begin{equation*}
        h[\mathcal{U}_\kappa \widetilde v]=\kappa^{-2}h[\widetilde v]<0.
    \end{equation*}
    Thus we can choose $\kappa$ large enough so that the support of $\mathcal{U}_\kappa \widetilde{v}$ lies in $(\ln R,\infty)$ and satisfies \eqref{eqn:lemproof}. Iterating this scaling argument, we can construct the desired sequence of disjointly supported functions (see e.g. \cite[Theorem XIII.6]{Reed1978IV:Operators}). 

    Now consider the function $v_L(t)=\sqrt{t}\left(1-\frac{\abs{\log t}}{\log L}\right)_+$ with $L> 1$, which is supported in $[1/L,L]$. Then one can calculate that 
    \begin{align*}
        \int_{0}^\infty \frac{(1+\varepsilon) \abs{v_L(t)}^2}{4t^2} \dd t=\frac{(1+\varepsilon)}{6}\ln L,
    \end{align*}
    whereas 
    \begin{align*}
        \int_0^\infty \abs{v_L^\prime(t)}^2 \dd t=\frac{1}{6}\ln L+\frac{2}{\ln L}.
    \end{align*}
    Thus fixing $L$ to be sufficiently large and taking $\widetilde{v}(t)=v_L(t)$ concludes the result. 
\end{proof}
Lemma \ref{lem:robinhardylog} and \ref{lem:reversehardy} will be enough to determine when our oscillatory potentials generate finitely or infinitely many negative eigenvalues. However, we note that there are general forms of Hardy's inequality which apply directly to such potentials. The following is generally attributed to Kats and Krein \cite{Kats1958CriteriaString}.
\begin{lemma}\label{thm:hille}
    Let $V\in L^1_{\mathrm{loc}}(\R_+)$, then for any $u\in C_0^\infty(\R_+)$ 
    \begin{equation*}
        \int_0^\infty V(r)\abs{u(r)}^2\dd r\leq 4\left(\sup_{t>0} t^{-1}\abs{\int_t^\infty V(s)\dd s}\right)\int_0^\infty \abs{u^\prime(r)}^2\dd r.
    \end{equation*}
\end{lemma}
\begin{proof}
    Take $W(t)\coloneqq \int_t^\infty V(s)\dd s$ and $\lambda\coloneqq\sup_{t>0}t^{-1}\abs{W(t)}$. Using integration by parts, $(|u|^2)^\prime=2\mathrm{Re}\left(u\overline{ u^\prime}\right)$ and Cauchy-Schwartz leads to
    \begin{align*}
        \int_0^\infty V\abs{u}^2\dd r\leq 2\mathrm{Re} \int_0^\infty W u\overline{u^\prime}\dd r
        &\leq 2\left(\int_0^\infty W^2\abs{u}^2\dd r \right)^{1/2}\left(\int_0^\infty \abs{u^\prime}^2\dd r\right)^{1/2}\\
        &\leq 2 \left(\int_0^\infty \frac{\lambda^2}{r^2}\abs{u}^2\dd r\right)^{1/2}\left(\int_0^\infty \abs{u^\prime}^2\dd r\right)^{1/2}.
    \end{align*}
    Applying the standard version of Hardy's inequality, \eqref{eqn:hardyweight}, produces the result. 
\end{proof}
Although the proof is identical to that in \cite{Kats1958CriteriaString}, we note that this inequality typically presupposes $V$ to be positive, except for an analogous formulation presented by Hille and Hartman \cite{Hille1948Non-oscillationTheorems,Hartman1964OrdinaryEquations}. This distinction leads, for instance, to the following bounds which state that for any $\alpha>0$
\begin{equation}\label{eqn:oschardyinequality}
    \int_{0}^\infty r^{\alpha-2}\sin(r^\alpha)|u(r)|^2\dd r\leq \frac{4}{\alpha}\int_{0}^\infty |u^\prime(r)|^2\dd r \text{ for all } u\in C_0^\infty(\R_+).
\end{equation}
These inequalities already indicate the critical nature of the oscillating potentials in Theorem \ref{thm:main}. However, comparing its statement with \eqref{eqn:oschardyinequality} we see that the constant fails to capture the exact coupling value for which these potentials generate finitely many negative eigenvalues. 

The insight offered by Lemma \ref{thm:hille} is that integral conditions, as opposed to pointwise ones, play a pivotal role in grasping the impact of oscillations. This emerges as a fundamental feature of our subsequent analysis.

\section{A ground state representation}\label{sec:groundstate}
In this section we introduce our main tool for studying Schr\"odinger operators with oscillating potentials. Since, we wish to deal with potentials of the type
\begin{equation}\label{eqn:oscillatingpotsec3}
    V(x)=\abs{x}^\beta \sin\abs{x}^\alpha(1+O(1)) \text{ as }\abs{x}\rightarrow\infty, 
\end{equation}
we begin by showing that the Schr\"odinger operators $-\Delta-V$ are well-defined. Subsequently we will show that its negative spectrum is discrete, even in the case where $\beta>0$. We note that this has been shown in \cite{Sasaki2007SchrodingerPotentials}, but we include the argument for the sake of completeness. For convenience, we operate under the assumption that our potentials are locally bounded. 
\begin{proposition}\label{prop:relbdd}
    Suppose that $V\in L^{\infty}_{\mathrm{loc}}(\R^d)$. If  
\begin{equation*}\label{eqn:lemrelformbdcond}
        \sup_{\omega\in \Sph^{d-1}}\abs{\int_r^\infty V(s\omega)\dd s}\rightarrow 0 \text{ as }r\rightarrow \infty,
    \end{equation*}
    then $V$ is form-bounded with respect to $-\Delta$, with relative bound zero.
\end{proposition}
\begin{proof}
    Let $u\in C^\infty_0(\R^d)$, then for any $R<\infty$,
    \begin{align*}
        \left\vert\int_{\R^d}V|u|^2\dd x\right\vert
        &\leq \norm[L^\infty(B_R)]{V}\int_{B_R}\abs{u}^2\dd x+\left\vert\int_{B_R^c}V|u|^2\dd x\right\vert.
    \end{align*}
    To bound the second term we work in spherical coordinates $(r,\omega)\in \R_+\times\Sph^{d-1}$, $r=|x|$ and $\omega=x/r$. Let $W(r\omega)\coloneqq \int_r^\infty V(s\omega)\dd s$, then
    \begin{align*}
        \left\vert\int_{B_R^c}V|u|^2\dd x\right\vert=&\abs{\int_{\Sph^{d-1}}\int_R^\infty V(r\omega)|u(r\omega)|^2 r^{d-1}\dd r\dd \omega}\\
        \leq &\int_{\Sph^{d-1}}\abs{\int_R^\infty W(r\omega)\partial_r(|u(r\omega)|^2)r^{d-1}\dd r}\dd \omega\\&+\int_{\Sph^{d-1}}\abs{W(R\omega)}|u(R\omega)|^2 R^{d-1}\dd \omega+\varepsilon_RR^{-1}(d-1)\norm{u}_{L^2(\R^d)}^2,
    \end{align*}
    where $\varepsilon_R\coloneqq \sup_{r>R,\omega\in\Sph^{d-1}}\abs{W(r\omega)}$. Label by $(1)$ and $(2)$ the first and second term in the last line.
    
    Using Cauchy-Schwartz and the trivial inequality $2ab\leq a^2+b^2$, for $a,b\geq 0$, we see that
    \begin{align*}
       (1)
        &\leq 2\int_{\Sph^{d-1}}\abs{\mathrm{Re}\int_R^\infty Wu\overline{\partial_ru} r^{d-1}\dd r}\dd \omega
        \\&\leq 2\int_{\Sph^{d-1}}\left(\int_R^\infty W(r\omega)^2\abs{u(r\omega)}^2r^{d-1}\dd r\right)^{1/2}\left(\int_R^\infty\abs{\partial_ru(r\omega)}^2 r^{d-1}\dd r\right)^{1/2}\dd \omega\\
        &\leq \varepsilon_R\int_{\Sph^{d-1}}\int_R^\infty\left(\abs{\partial_r u}^2+\abs{u}^2 \right) r^{d-1}\dd r\dd \omega.
    \end{align*}
    Then for the second term,  
    \begin{align*}
        (2) &\leq \varepsilon_R\int_{\Sph^{d-1}}\abs{\int_{R}^\infty \partial_r\left(\abs{u(r\omega)}^2r^{d-1}\right)\dd r}\dd \omega\\
        &\leq \varepsilon_R\int_{\Sph^{d-1}}\abs{\int_{R}^\infty \partial_r\left(\abs{u(r\omega)}^2\right)r^{d-1}\dd r}+\varepsilon_RR^{-1}(d-1)\int_R^\infty \abs{u(r\omega)}^2r^{d-1}\dd r\dd \omega,
    \end{align*}
    which we can bound by using the same approach as for $(1)$.
    
    Putting this together we find that 
    \begin{align*}
        \left\vert\int_{\R^d}V|u|^2\dd x\right\vert
        &\leq 2\varepsilon_R\int_{\R^d}\abs{\nabla u}^2\dd x+\left(2(d-1)\varepsilon_RR^{-1}+2\varepsilon_R+\norm[L^\infty(B_R)]{V}\right)\int_{\R^d}\abs{u}^2\dd x.
    \end{align*}
    Under the given assumptions of $V$, we can choose $R$ so that $\varepsilon_R$ is arbitrarily small. This concludes the result.
\end{proof}
Then it is clear the operators considered in Theorem \ref{thm:main} are well-defined in the sense of quadratic forms with $C_0^\infty(\R^d)$ as their form core. This follows more generally for the potentials \eqref{eqn:oscillatingpotsec3} with $\alpha>0$ and $\alpha-\beta>1$, since for any $\omega\in \Sph^{d-1}$
\begin{align*}
    \int_r^\infty V(s\omega)\dd s=\alpha^{-1}r^{1+\beta-\alpha}\cos r^\alpha(1+O(1))\text{ as }r\rightarrow \infty.
\end{align*}
\begin{proposition}\label{prop:essspectrum}
    Let $V\in L^\infty_{\mathrm{loc}}(\R^d)$ satisfy the condition of Proposition \ref{prop:relbdd}, then the essential spectrum of $-\Delta-V$ in $L^2(\R^d)$ coincides with $[0,\infty)$.
\end{proposition}
\begin{proof}
    By Weyl's Theorem it suffices to show that for any sequence $\{u_k\}_{k=1}$ which converges weakly to zero in $H^1(\R^d)$ that $\int_{\R^d}V\abs{u_k}^2\dd x\rightarrow 0$. 
    
    We have shown in the proof of Proposition \ref{prop:relbdd} that for any $\varepsilon>0$ we can find $R<\infty$ such that for all $u\in H^1(\R^d)$
    \begin{align*}
        \left\vert\int_{\R^d}V|u|^2\dd x\right\vert
        &\leq \varepsilon\int_{\R^d}\abs{\nabla u}^2\dd x+\varepsilon\int_{\R^d}\abs{u}^2\dd x+\norm[L^\infty(B_R)]{V}\int_{B_R}|u|^2\dd x.
    \end{align*}
    Thus, it follows that
    \begin{align*}
        \limsup_{k\rightarrow \infty}\left\vert\int_{\R^d}V|u_k|^2\dd x\right\vert
        &\leq \varepsilon\limsup_{k\rightarrow \infty}\norm[H^1(\R^d)]{u_k}\leq \varepsilon\sup_{k\geq 1}\norm[H^1(\R^d)]{u_k},
    \end{align*}
    where we have used that $\chi_{B_R}u_k\rightarrow 0$ in $L^2(B_R)$ (see \cite[Proposition 2.36]{Frank2022SchrodingerInequalities}). Since we can bound the supremum on the right by the Banach-Steinhaus Theorem and choose any $\varepsilon>0$ the assertion holds.
\end{proof}

Now we are ready to introduce a ground state representation for the operators we have just defined. The approach we detail was used by Combescure and Ginibre in \cite{Combescure1976SpectralPotentials,Combescure1980SpectralPotentials}, where they also investigated oscillating potentials. Among their results is a version of the three-dimensional Birman--Schwinger bound for $N_0(V)$ in terms of a function $W$ which satisfies $\nabla\cdot W=-V$. However, they did not consider more qualitive conditions for the finiteness of $N_0(V)$. 

Consider the operator $H_0\coloneqq \lapr-V \text{ in }L^2(\R_+)$
with Dirichlet boundary conditions at zero. Let $W$ be a measurable function satisfying $W^\prime=-V$, e.g. $W(r)=\int_r^{\infty} V(s)\dd s$. Under the assumption that $V\in L^\infty_{\mathrm{loc}}(\R^d)$ and $W(r)$ decays to zero, $H_0$ is well-defined, as demonstrated in Proposition \ref{prop:relbdd}. 

If we introduce the operator $D=\frac{\mathrm{d}}{\mathrm{d}x}-W$ with domain $H_0^1(\R_+)$. Then, the operator $\widetilde H_0\coloneqq H_0+W^2$ factorises as $D^\ast D=\widetilde H_0$. To see this, note that for any $u\in H^1_0(\R_+)$,
\begin{equation*}
    (D^\ast Du,u)_{L^2(\R_+)}=\norm{D u}_{L^2(\R_+)}^2=\int_0^\infty |u^\prime-W u|^2\dd r,
\end{equation*}
which, by the expansion
\begin{align*}
    |u^\prime-W u|^2&=|u^\prime|^2-u^\prime \overline{W u}-Wu\overline{u^\prime}+W^2|u|^2\\
    &=|u^\prime|^2-W(|u|^2)^\prime+W^2|u|^2,
\end{align*}
leads to the equality in the sense of quadratic forms.

To obtain a ground state representation, suppose that $U$ is some measurable function satisfying $U^\prime =W$, e.g. $U(x)=\int_0^r W(s)\dd s$. Then $e^{U}$ serves as an effective ground state for $\widetilde H_0$, leading to
\begin{align*}
    (\widetilde H_0 e^{U}u,e^{U}u)_{L^2(\R_+,\dd r)}=\int_0^\infty|(ue^{U})^\prime-(e^{U})^\prime u|^2\dd r=\int_0^\infty |u^\prime|^2 e^{2U}\dd r.
\end{align*}
Consider the unitary transformation $\mathcal{U}\colon  L^2(\R_+,e^{2U}\dd r)\rightarrow L^2(\R_+,\dd r)$ given by $\mathcal{U}u=e^{U}u$. Then we have demonstrated  that
\begin{equation*}
    \mathcal{U}^{-1}H_0 \ \mathcal{U}=-e^{-2U}\frac{\mathrm{d}}{\mathrm{d}r}e^{2U}\frac{\mathrm{d}}{\mathrm{d}r}-W^2 \text{ in }L^2(\R_+,e^{2U}\dd r),
\end{equation*}
where the right side is a Sturm-Liouville operator with Dirichlet conditions at $0$.

If, more generally, we consider the operator $H_{\sigma,R}=\lapr-V$ on $L^2(R,\infty)$ with Robin boundary conditions, $u^\prime(R)-\sigma u(R)=0$. Then under the same unitary transform it follows that
\begin{equation*}
    \mathcal{U}^{-1}H_{\sigma,R}\ \mathcal{U}=-e^{-2U}\frac{\mathrm{d}}{\mathrm{d}r}e^{2U}\frac{\mathrm{d}}{\mathrm{d}r}-W^2 \text{ in }L^2((R,\infty),e^{2U}\dd r),
\end{equation*}
where the Sturm-Liouville operator has Robin boundary conditions with coefficient $\widetilde{\sigma}=e^{2U(R)}(\sigma-W(R))$. Namely it corresponds to the quadratic form
\begin{equation*}
    \int_R^\infty (|u^\prime|^2-W^2|u|^2) e^{2U}\dd r+\widetilde{\sigma}\abs{u(R)}^2.
\end{equation*}

Now in application, suppose that $U$ decays to zero at infinity. Then the operators above become asymptotically equivalent to a Schr\"odinger operator with potential $W^2$. Then, seemingly we can apply conditions like \eqref{eqn:finitecond} and \eqref{eqn:infinitecond} to $W^2$. The core of the subsequent theorem is to employ and iterate this idea twice.
\newpage
\begin{theorem}\label{thm:criticcoupl}
Let \( V \in L^\infty_{\mathrm{loc}}(\R^d) \) satisfy
\begin{equation}\label{eqn:thmpotentialscrit}
V(x) = \lambda |x|^\beta \sin |x|^\alpha + o((|x|\ln|x|)^{-2}) \quad \text{as } |x| \to \infty,
\end{equation}
where \( \alpha > 0 \), \( \alpha-\beta > 1 \), and \( 2\alpha-\beta > 2 \). Then the negative spectrum of \( -\Delta - V \):
\begin{enumerate}[left=0.5cm]
    \item Consists of finitely many negative eigenvalues if either of the following holds:
    \begin{enumerate}[left=0.5cm]
        \item \( \alpha-\beta > 2 \) for any \( \lambda \in \R \),
        \item \( \alpha-\beta = 2 \) and \( |\lambda| \leq \alpha |d-2|/\sqrt{2} \).
    \end{enumerate}
    \item Consists of infinitely many negative eigenvalues, accumulating at zero, if either of the following holds:
    \begin{enumerate}[left=0.5cm]
        \item \( \alpha-\beta < 2 \) for any \( \lambda \in \R \backslash \{0\} \),
        \item \( \alpha-\beta = 2 \) and \( |\lambda| > \alpha |d-2|/\sqrt{2} \).
    \end{enumerate}
\end{enumerate}
\end{theorem}

\begin{proof}
    For any $\varepsilon\in (0,1/4)$ we can choose $R<\infty$ sufficiently large such that 
    \begin{align}\label{eqn:asymproofthm}
        |V(x)-\lambda\abs{x}^\beta\sin\abs{x}^\alpha|<\frac{\varepsilon}{3}(\abs{x}\ln\abs{x})^{-2} \text{ for all }|x|>R. 
    \end{align}
    
    We start by showing $(1)$. Using \eqref{eqn:asymproofthm} we can bound the operator from below by replacing $V$ with
    \begin{equation*}
        \chi_{B_R}(x)V+\chi_{B_R^c}(x)\left(\lambda\abs{x}^\beta\sin\abs{x}^\alpha+\frac{\varepsilon}{3}(\abs{x}\ln\abs{x})^{-2}\right).
    \end{equation*}
    Then, following the argument in Section \ref{sec:hardy} we reduce this operator further by imposing Robin boundary conditions on $\partial B_R$. Let $H_{\sigma,R}^{-}$ and $H_{\sigma,R}^{+}$ denote the respective restrictions of this reduced operator on $L^2(B_R)$ and $L^2(\overline{B_R}^c)$ with corresponding forms \eqref{eqn:robinops}. Then 
    \begin{equation*}
        N_0(V)\leq N_0(H_{\sigma,R}^{-})+N_0(H_{\sigma,R}^{+}),
    \end{equation*}
    where from $V\in L^\infty_{\mathrm{loc}}(\R^d)$ it follows that $N_0(H_{\sigma,R}^{-})<\infty$. 
    
    Changing to polar coordinates we can use separation of variables in the eigenbasis of $-\Delta_{\Sph^{d-1}}$ corresponding to the eigenvalues $\{\Lambda_k\}_{k=1}$. It follows that $H^+_{R,\sigma}$ can be written as the direct sum
    \begin{align*}
        \bigoplus_{k=1}\left(-\frac{\mathrm{d}^2}{\mathrm{d}r^2}+\frac{4\Lambda_k+(d-1)(d-3)}{4r^2}-\lambda r^\beta\sin r^\alpha- \frac{\varepsilon}{3}(r\ln r)^{-2} \right)
    \end{align*}
    where each are considered on $L^2((R,\infty),\dd r)$ with Robin boundary coefficient $\widetilde\sigma\coloneqq (1-d)/2R+R^{1-d}\sigma$. We denote each of these by $h^{(k)}_{\sigma,R}$ and note that
    \begin{equation*}
        N_0\left(H^{+}_{\sigma,R}\right)= \sum_{k=1} N_0\left(h^{(k)}_{\sigma,R}\right).
    \end{equation*}
    
    Now we reduce each of these operators using the ground state representation above. Most importantly, we apply it only with respect to the oscillatory part of the potential. For $\alpha-\beta>1$, as $r\rightarrow \infty$, we have 
    \begin{align*}
        W(r)=\int_r^\infty \lambda s^\beta\sin s^\alpha\dd s=\frac{\lambda \cos r^\alpha}{ \alpha r^{\alpha-\beta-1}}+O(r^{1+\beta-2\alpha}),
    \end{align*}
    and if $2\alpha-\beta>2$, then 
    \begin{align*}
        U(r)=-\int_r^\infty W(s)\dd s=\frac{\lambda \sin r^\alpha}{\alpha^2 r^{2\alpha-\beta-2}}+O(r^{2+\beta-3\alpha}).
    \end{align*}
    From this asymptotic behaviour, we can enlarge $R$ so that each $h^{(k)}_{\sigma,R}$ is unitarily equivalent, via $\mathcal{U}\phi=e^{U}\phi$, to an operator which can be bounded from below by
    \begin{equation*}
        \widetilde{h}^{(k)}_{\sigma,R}\coloneqq -\frac{\mathrm{d}^2}{\mathrm{d}r^2}+\frac{4\Lambda_k+(d-1)(d-3)}{4r^2}-\frac{\lambda^2 (\cos r^\alpha)^2}{ \alpha^2 r^{2\alpha-2\beta-2}}-\frac{2\varepsilon}{3}(r\ln r)^{-2}.
    \end{equation*}
    Where the above are considered in $L^2(R,\infty)$ with Robin boundary coefficient $\widetilde{\sigma}-W(R)$. Note that we have moved the lower order parts of $W$ and $U$ onto the logarithmic term. Thus
    \begin{equation*}
        N_0(H_{\sigma,R}^+)\leq \sum_{k=1}N_0\left(\widetilde{h}^{(k)}_{\sigma,R}\right).
    \end{equation*}

    At this stage if $\alpha-\beta>2$ then by choosing $R$ and $\sigma$ suitably large, we can apply Lemma \ref{lem:robinhardylog} with $\rho=0$ to each of these operators uniformly. Indeed we only have to check $k=1$ where $\Lambda_1=0$. Then each generates no negative eigenvalues and it follows that $N_0(V)<\infty$. 
    
    For $(1)$, it remains to consider the case where $\alpha-\beta=2$.  To do so, we repeat the argument with respect to the operators $\widetilde{h}_{\sigma,R}^{(k)}$. Note that
    \begin{align*}
        \widetilde{W}(r)&\coloneqq \int_r^\infty \frac{\lambda^2\cos(s^\alpha)^2}{\alpha^2 s^{2}}\dd s=\frac{\lambda^2}{\alpha^2}\left(\frac{1}{2r}\right)+O(r^{-(\alpha+1)}),
    \end{align*}
    and
    \begin{align*}
        \widetilde{U}(r)\coloneqq \int_R^r\widetilde{W}(s)\dd s=\frac{\lambda^2}{2\alpha^2}
                \ln (r/R)+O(1).
    \end{align*}
    Thus, after taking $R$ to be sufficiently large and applying the ground state representation, we see that each of the operators $\widetilde{h}^{(k)}_{\sigma,R}$ are unitarily equivalent to an operator bounded from below by
    \begin{equation*}
        r^{-\lambda^2/\alpha^2}\frac{\mathrm{d}}{\mathrm{d}r}r^{\lambda^2/\alpha^2}\frac{\mathrm{d}}{\mathrm{d}r}+\frac{4\Lambda_k+(d-1)(d-3)}{4r^2}-\frac{\lambda^4}{4\alpha^4 r^2}-\varepsilon(r\ln r)^{-2},
    \end{equation*}
    in $L^2((R,\infty),r^{\lambda^2/\alpha^2}\dd r)$ with Robin boundary conditions $(\widetilde{\sigma}-W(R)-W^2(R))$. Then after choosing $\sigma$ suitably large we can apply Lemma \ref{lem:robinhardylog} with $\rho=\lambda^2/\alpha^2$. It follows that the operators are uniformly positive if $\lambda^2/\alpha^2\leq (d-2)^2/2$. Which concludes the statement $(1)$.

    To prove $(2)$ we apply similar reasoning. Starting with \eqref{eqn:asymproofthm} we can bound the operator from above by making the potential smaller. Then we use bracketing and study the part of this operator restricted to $\overline{B_R}^c$ with Dirichlet conditions. If we denote this Dirichlet operator by $H^{+}_{\infty,R}$ then it follows that 
    \begin{equation*}
        N_0(V)\geq N_0(H^{+}_{\infty,R}). 
    \end{equation*}
    Using separation of variables, this operator can be decomposed into the direct sum of
    \begin{align*}
        \bigoplus_{k=1}\left(-\frac{\mathrm{d}^2}{\mathrm{d}r^2}+\frac{4\Lambda_k+(d-1)(d-3)}{4r^2}-\lambda r^\beta\sin r^\alpha+\frac{\varepsilon}{3}(r\ln r)^{-2} \right),
    \end{align*}
    each with Dirichlet boundary conditions on $L^2(R,\infty)$. 

    In the case where $\alpha-\beta=2$, we use the ground state representation twice in the same way as before. Then each of the operators can be bound from above by
    \begin{equation*}
        r^{-\lambda^2/\alpha^2}\frac{\mathrm{d}}{\mathrm{d}r}r^{\lambda^2/\alpha^2}\frac{\mathrm{d}}{\mathrm{d}r}+\frac{4\Lambda_k+(d-1)(d-3)}{4r^2}-\frac{\lambda^4}{4\alpha^4 r^2}+\varepsilon(r\ln r)^{-2},
    \end{equation*}
    in $L^2((R,\infty),r^{\lambda^2/\alpha^2}\dd r)$ with Dirichlet conditions. Then if $\lambda^2/\alpha^2>(d-2)^2/2$ we can apply Lemma \ref{lem:reversehardy} to show that the operator corresponding to $k=1$ produces an infinite number of negative eigenvalues, proving statement $(b)$.
    
    Then it remains to prove $(2a)$ where $\alpha-\beta<2$. In this case we apply the ground-state representation twice with a modification to $\widetilde{W}$ in the second step. If $\alpha-\beta\neq 3/2$ then we have that 
    \begin{align*}
        \widetilde{W}(r)=\frac{\lambda^2}{\alpha^2}\left(\frac{1}{\abs{4 \alpha-4 \beta-6}}\right)r^{-2\alpha+2\beta+3}(1+o(1))
    \end{align*}
    where we either take $\int_r^\infty \cdot \dd s$ or $-\int_R^r \cdot \dd s$ with the integrand as above. Whereas if $\alpha-\beta=3/2$ we use the latter form and find that
    \begin{align*}
        \widetilde{W}(r)=\frac{\lambda^2}{2\alpha^2}\ln(r)(1+o(1)). 
    \end{align*}
    In either case we take $\widetilde{U}(r)=\int_R^r \widetilde{W}(s)\dd s$ as before. 
    
    Consider just the $k=1$ component of the Dirichlet operator. Then from the above it is unitarily equivalent to an operator with quadratic form 
    \begin{equation*}
        \int_R^\infty \left(\abs{u^\prime}^2+\left(\frac{(d-1)(d-3)}{4r^2}+\varepsilon(r\ln r)^{-2}\right)\abs{u}^2-\widetilde{W}(r)^2\abs{u}^2\right)e^{2\widetilde{U}(r)}\dd r.
    \end{equation*}
    The leading term of the potential $\widetilde{W}(r)^2$ is strictly positive and asymptotically homogeneous of degree strictly larger than $-2$. Then for any positive $u\in C_0^\infty(R,\infty)$ we can scale it to ensure that the form above is strictly negative (see \cite[Theorem XIII.6]{Reed1978IV:Operators}). Similar to the proof of Lemma \ref{lem:reversehardy} this leads to an infinite dimensional subspace of $L^2((R,\infty),e^{2\widetilde{U}}\dd r)$ for which the form is negative. This concludes the proof of the statement $(2a)$. 
\end{proof}

We finish this section with a couple of remarks. 

\begin{remark}    
    In the statement of the theorem we can consider potentials which oscillate like $\lambda \abs{x}^\beta\sin\mu\abs{x}^\alpha$. Scaling shows that the statement remains the same where in the critical case the conditions on $\abs{\lambda}$ are substituted with those on $\abs{\lambda/\mu}$. 
    
    Another generalisation can be made with respect to the error $o((\abs{x}\ln\abs{x})^{-2})$. We can add to $V$ in \eqref{eqn:thmpotentialscrit} any $o(\cdot)$ correction of the form
    \begin{equation*}
        \frac{\abs{x}^\beta \sin \eta \abs{x}^\alpha}{\ln \abs{x}}, \ \eta>0,
    \end{equation*}
    without changing the result. This follows by encorporating this term in the ground state representation. 
\end{remark}

\begin{remark}
    For the other rapidly oscillating example alluded to in \cite{Raikov2016DiscretePotentials,Sasaki2007SchrodingerPotentials} which asymptotically looks like 
    \begin{equation*}
         V(x)=\lambda e^{\abs{x}}\abs{x}^{-2}\sin\left(e^{\abs{x}}\right) \text{ as }\abs{x}\rightarrow \infty,
    \end{equation*}
    it is clear from the above that it has finitely many negative eigenvalues for any coupling $\lambda$. Indeed, it is apparent that our methods would apply to a more general class of rapidly oscillating potentials.
\end{remark}

\section{Proof of Theorem \ref{thm:main}}\label{sec:mainproof}

In this section, we combine the tools from the previous two sections with the method used by Hassell and Marshall in \cite{Hassell2008Eigenvalues-2} to prove Theorem \ref{thm:main}. Specifically, we will employ the Sturm oscillation theorem (see e.g. \cite{SimonSturm}) which was used originally in the works \cite{Wong1969OscillationCoefficients,Willett1969OnEquations} to establish the critical nature of oscillating potentials in one dimension.

\begin{proof}[Proof of Theorem \ref{thm:main}]
For the case where $\abs{\lambda} \leq \alpha\abs{d-1}/\sqrt{2}$, the result is immediate from Theorem \ref{thm:criticcoupl}. We henceforth fix $\abs{\lambda}>\alpha\abs{d-1}/\sqrt{2}$.

Following the approach in Theorem \ref{thm:criticcoupl} we can bound $-\Delta-V$ from below by the direct sum of operators $H_{\sigma,R}^{-}\oplus H_{\sigma,R}^{+}$ with Robin boundary conditions. Moreover, for any $\varepsilon>0$ we can select $R<\infty$ such that $N_0(H_{\sigma,R}^{-})<\infty$ and $H_{\sigma,R}^{+}$ can be bounded below by the direct sum of 
\begin{equation*}
        \bigoplus_{k=1}\left(-r^{-\lambda^2/\alpha^2}\frac{\mathrm{d}}{\mathrm{d}r}r^{\lambda^2/\alpha^2}\frac{\mathrm{d}}{\mathrm{d}r}+\frac{4\Lambda_k+(d-1)(d-3)}{4r^2}-\frac{\lambda^4}{4\alpha^4 r^2}-\varepsilon(r\ln r)^{-2}\right), 
    \end{equation*}
each in $L^2((R,\infty),r^{\lambda^2/\alpha^2}\dd r)$ with certain Robin boundary conditions. Note that eventually there is a $K\in \N$ for which $4\Lambda_k+(d-1)(d-3)\geq \lambda^4/\alpha^4$ for all $k> K$. Therefore, by Lemma \ref{lem:robinhardylog} only the first $K$ operators in the direct sum produce any negative eigenvalues.

For each of these operators we can use a variational trick and swap out the Robin boundary conditions for Dirichlet conditions. This only amounts to a rank-one change (see for instance \cite{SimonTrace}). Denoting these operators by $h^{(k)}_{R}$ we have that
\begin{equation*}
    N_E(V)\leq N_0(H^{-}_{\sigma,R})+ \sum_{k=1}^{K}N_E(h^{(k)}_{R})+K.
\end{equation*}

To calculate the asymptotic behaviour of each $N_E(h^{(k)}_{R})$ as $E\downarrow 0$ we use the Sturm oscillation theorem. Fixing $k$, $N_E(h^{(k)}_{R})$ coincides with the number of zeroes of any $u$ satisfying
\begin{equation*}\label{eqn:sturmproofthm1}
      h^{(k)}_{R}u(r)=-Eu(r) \text{ for  } r>R. 
\end{equation*}
Introducing variables $\rho=\lambda^2/\alpha^2$ and $\eta_k=\left(2\rho-(d-2)^2-4\Lambda_k\right)/4>0$, and taking $v(r)\coloneqq r^{(\rho-1)/2}u(r)$ the equation transforms to
\begin{align*}
    - r^2 v''(r)- r v'(r)-\left(\eta_k-Er^2+\varepsilon(\ln r)^{-2} \right)v(r) =0.
\end{align*}
Finally, taking $\widetilde{v}(\sqrt{\eta_k}t)=v(e^{t})$ this changes to 
\begin{align*}
    - \widetilde{v}''(t)-Q(t)\widetilde{v}(t) =0 \text{ for }t>\ln R,
\end{align*}
where $Q(t)=\eta_k-E e^{2t}+\varepsilon/t^2$.

The zeros of the solution $\widetilde{v}$ are then determined by using Pr\"ufer variables, namely $\theta(t)$, defined as 
\begin{equation*}
    \tan \theta(t)=\sqrt{\eta_k} \frac{\widetilde{v}(t)}{\widetilde{v}^\prime(t)}.
\end{equation*}
Combining with the equation above we see that 
\begin{equation*}
    (\sec\theta(t))^2\theta'(t)=\sqrt{\eta_k}-\sqrt{\eta_k}\frac{\widetilde{v}(t)\widetilde{v}^{\prime\prime}(t)}{(v'(t))^2}=\sqrt{\eta_k}+\sqrt{\eta_k} Q(t)\frac{\widetilde{v}(t)^2}{(\widetilde{v}'(t))^2},
\end{equation*}
and thus by definition of $\theta$,
\begin{equation*}
    \theta'(t)=Q(t)/\sqrt{\eta_k}+\sqrt{\eta_k}(1-Q(t)/\eta_k)\cos(\theta(t))^2.
\end{equation*}
Given Dirichlet conditions at $t=\ln R$ we take $\theta(\ln R)=0$. Then, whenever $\widetilde{v}(t)=0$ it follows that $\theta^\prime(t)=1$. Thus the number of zeroes of $\widetilde{v}$ in any interval $[a,b]$ is given by $(\theta(b)-\theta(a))/\pi$.

Notice that if $Q(t)< 0$ for all $t>t_0$ then $v(t)$ is no longer oscillatory. If $v$ has a local minimum or maximum at some $\tau> t_0$, then $\cos(\theta(\tau))=0$ and $\theta'(\tau)=Q(\tau)/\sqrt{\eta_k}<0$. Thus there can only be one additional zero past $t_0$,  since $\theta$ cannot move beyond $\tau$.  Therefore we need only calculate $t_0$ and the zeroes in $[\ln(R),t_0]$.

Solving $Q(t_0)=0$ leads to 
\begin{equation*}
    t_0=\frac{1}{2} \abs{\ln E}+O(1)\text{ as }E\downarrow 0.
\end{equation*}
Then the number of zeroes of $\widetilde v$ is equal to $(\theta(t_0)-\theta(\ln R))/\pi+O(1)$ and
\begin{align}\label{eqn:zeroestheta}
    \theta(t_0)-\theta(\ln R)&=\int_{\ln R}^{t_0}\theta^\prime(s)\dd s=\int_{\ln R}^{t_0} \sqrt{\eta_k}\dd s+O(1)=\frac{\sqrt{\eta}}{2} \abs{\ln E}+O(1).
\end{align}
Where we have used that as $E\downarrow 0$, 
\begin{align*}
    \int_{\ln R}^{t_0}Ee^{2t}\dd s\leq \frac{Ee^{2t_0}}{2}\lesssim 1,
\end{align*}
which in addition with other components of $\theta^\prime$, e.g. from the term $\varepsilon/t^2$, leads to the $O(1)$ remainder in \eqref{eqn:zeroestheta}.

Then we have shown that
\begin{equation*}
    N_E(h_R^{(k)})=\frac{\sqrt{\eta_k}}{2\pi}\abs{\ln E}=(2\pi)^{-1}\abs{\ln E}\sqrt{\frac{\lambda^2}{2\alpha^2}-\frac{(d-2)^2}{4}-\Lambda_k}+O(1) \text{ as }E\downarrow 0
\end{equation*}
which leads, by definition of $K$, to the desired sum 
\begin{equation*}\label{eqn:NEVasmp}
    N_E(V)\leq (2\pi)^{-1}\abs{\ln E}\sum_{k=1}\sqrt{\left(\frac{\lambda^2}{2\alpha^2}-\frac{(d-2)^2}{4}-\Lambda_k\right)_+}+O(1) \text{ as }E\downarrow 0.
\end{equation*}

To find the identical lower bound we can apply exactly the same argument starting with a direct sum of Dirichlet operators like in the proof of Theorem \ref{eqn:mainthmform}. Then the only difference in calculating the eigenvalues of each component is the sign of the error $\varepsilon(r \ln r)^{-2}$. Since this doesn't affect the subsequent analysis we arrive precisely at the desired result.
\end{proof}

We finish by proving our statement regarding the negative eigenvalues which accumulate at zero in the critical case. 
\begin{proof}[Proof of Corollary \ref{cor:main}]
    Let $M$ denote the coefficient of $\abs{\ln E}$ in the formula \eqref{eqn:mainthmform}. The negative eigenvalues $\abs{\lambda_k}\rightarrow 0$ as $k\rightarrow \infty$ and, by definition, $N_{\abs{\lambda_k}}(V)=k-1$. Therefore it follows from the formula \eqref{eqn:mainthmform} that there exists $K\geq 1$ and $C<\infty$ such that
    \begin{equation*}
        \abs{(k-1)+M\ln \abs{\lambda_k}}\leq C \text{ for all }k\geq K.
    \end{equation*}
    Thus it follows that 
    \begin{equation*}
       \exp\left(-\frac{k}{M}\right)\exp\left(\frac{1-C}{M}\right)\leq \abs{\lambda_k}\leq \exp\left(-\frac{k}{M}\right)\exp\left(\frac{1+C}{M}\right) \text{ as }k\rightarrow \infty,
    \end{equation*}
    which concludes the result.
\end{proof}

\subsection*{Acknowledgements}
This work was funded by the Deutsche Forschungsgemeinschaft (DFG) project TRR 352 – Project-ID 470903074. The author is grateful to Rupert L. Frank for his guidance and for the introduction to this problem.

\end{document}